\documentclass{amsart}
\usepackage{graphicx}
\usepackage{color}
\newtheorem{thm}{Theorem}[section]
\newtheorem{cor}[thm]{Corollary}
\newtheorem{lem}[thm]{Lemma}
\newtheorem{prop}[thm]{Proposition}
\theoremstyle{definition}
\newtheorem{defn}[thm]{Definition}
\theoremstyle{example}
\newtheorem{exam}[thm]{Example}
\theoremstyle{remark}
\newtheorem{rem}[thm]{Remark}
\numberwithin{equation}{section}

\begin{document}

\title[Gram matrix associated to controlled frames]
{Gram matrix associated to controlled frames}
\author{E. Osgooei}
\address{Faculty of sciences\\
Urmia University of Technology\\ Urmia, Iran.} \email{
e.osgooei@uut.ac.ir; osgooei@yahoo.com} \dedicatory{}
\author{A. Rahimi}
\address{Department of sciences\\
 University of Maragheh\\ Maragheh, Iran.}
\email{rahimi@maragheh.ac.ir} \subjclass[2010]{42C15; 42C40; 94A12;
68M10.}

\keywords{controlled frames; Gram matrix; Gram operators; controlled Riesz bases. }%

\begin{abstract}
Controlled frames have been recently introduced in Hilbert spaces to
improve the numerical efficiency of interactive algorithms for
inverting the frame operator. In this paper, unlike the cross-Gram
matrix of two different sequences which is not always a diagnostic
tool, we define the controlled-Gram matrix of a sequence as a
practical implement to diagnose that a given sequence is a
controlled Bessel, frame or Riesz basis. Also, we discuss the cases
that the operator associated to controlled Gram matrix will be
bounded, invertible, Hilbert-Schmidt or a trace-class operator.
Similar to standard frames, we present an explicit structure  for
controlled Riesz bases and show that every $(U, C)$-controlled Riesz
basis $\{f_{k}\}_{k=1}^{\infty}$ is in the form
$\{U^{-1}CMe_{k}\}_{k=1}^{\infty}$, where $M$ is a bijective
operator on $H$. Furthermore, we propose an equivalent accessible condition to the
sequence $\{f_{k}\}_{k=1}^{\infty}$ being a $(U, C)$-controlled
Riesz basis.
\end{abstract}
\maketitle
\section{INTRODUCTION}
Frames in Hilbert spaces were first introduced by Duffin and
Schaeffer to deal with nonharmonic Fourier series in 1952 \cite{2}
and widely studied from 1986 since the great work by Daubechies, Grossmann and Meyer
constructed \cite{3}. Nowadays frames play an important role in pure
and applied mathematics, also have many applications in signal
processing \cite{4}, coding and communications \cite{5}, filter bank
theory \cite{6}. We refer to \cite{12, 13} for an introduction to
frame theory and its applications.
\par Controlled frames as a generalization of frames, have been introduced for getting an improved
solution of a linear system of equation $Ax=B$, which this system
can be solved by equation $PAx=PB$, where $P$ is a suitable matrix
to get a better duplicate algorithm \cite{1}. Controlled frames used earlier just as a tool for spherical wavelets  and the
relation between controlled frames and standard frames were
developed in \cite{10}. The main advantage of these frames lies in the fact that they retain all the advantages of standard frames but additionally they give a generalized way to check the frame condition while offering a numerical advantage in the sense of preconditioning. Recent developments in this direction can be found in  \cite{HH,KM,11,Rah3,Rah4,Rah6} and the references therein.

\par A sequence $\{f_{k}\}_{k\in I}\subseteq H$ is a frame for $H$ if
there exist $0<A\leq B<\infty$ such that
\begin{equation}\label{ham}
A\|f\|^{2}\leq\sum_{k\in I}|\langle f, f_{k}\rangle|^{2}\leq
B\|f\|^{2},\ \ f\in H.
\end{equation}
The constants $A$ and $B$ are called lower and upper frame bounds,
respectively. The sequence $\{f_{k}\}_{k\in I}\subseteq H$ is a
Bessel sequence for $H$, if the right hand inequality in
(\ref{ham}), holds for all $f\in H.$
\\Let $\{f_{k}\}_{k\in I}$ be a Bessel sequence for $H$. Then the
operator
\begin{equation*}
T:\ell^{2}(\mathbb{N})\rightarrow H,\ \ T\{c_{k}\}_{k\in
I}=\sum_{k\in I}c_{k}f_{k},
\end{equation*}
is called the synthesis operator and its adjoint
\begin{equation*}
T^{*}:H\rightarrow\ell^{2}(\mathbb{N}),\ \ T^{*}f=\{\langle f,
f_{k}\rangle\}_{k\in I},
\end{equation*}
is called the analysis operator of $\{f_{k}\}_{k\in I}$. By composing the operators $T$ and $T^*$, we get the frame operator $S=TT^*$, which is a bounded, positive, invertible operator and $AI\leq S\leq BI.$ 
\par A Riesz basis for $H$ is a family of the form
$\{Ue_{k}\}_{k=1}^{\infty}$, where $\{e_{k}\}_{k=1}^{\infty}$ is an
orthonormal basis for $H$ and $U$ is a bounded bijective operator on
$H$.
\\Checking equation (\ref{ham}) is not always an easy task in
practice. So the conditions for a sequences $\{f_{k}\}_{k\in I}$
being a Bessel sequence, frame, or Riesz basis can be expressed in
terms of the so-called Gram matrix.
\subsection{Gram matrix of discrete frames}
 If $\{f_{k}\}_{k=1}^{\infty}$ is a Bessel sequence, we can
compose the synthesis operator $T$ and its adjoint $T^{*}$ to obtain
the bounded operator
$$T^{*}T:\ell^{2}(\mathbb{N})\rightarrow\ell^{2}(\mathbb{N});\ \
T^{*}T\{c_{k}\}_{k=1}^{\infty}=\left\{\langle\sum_{\ell=1}^{\infty}c_{\ell}f_{\ell},
f_{k}\rangle\right\}_{k=1}^{\infty}.$$
This operator is called the Gram operator on $\ell^2(\mathbb{N})$
associated to $\{f_{k}\}_{k=1}^{\infty}$ and corresponds to a matrix
given by
$$T^{*}T=\{\langle f_{k}, f_{j}\rangle\}_{j, k=1}^{\infty}.$$
The matrix $\{\langle f_{k}, f_{j}\rangle\}_{j, k=1}^{\infty}$ is
called the matrix associated with $\{f_{k}\}_{k=1}^{\infty}$ or Gram
matrix.
\par The ability of combining the synthesis and analysis operators of a Bessel sequence
to make a sensitive operator is essential in frame theory and its
applications. For example in \cite{7}, for given two different
Bessel sequences $\{f_{k}\}_{k\in I}$ and $\{g_{k}\}_{k\in I}$ the
synthesis operator of $\{f_{k}\}_{k\in I}$ with the analysis
operator of $\{g_{k}\}_{k\in I}$ is composed and a fundamental
operator is generated. This operator is called the cross-Gram
operator associated with the sequence $\{\langle f_{k},
g_{j}\rangle\}_{j, k=1}^{\infty}$ \cite{8, 9} and the conditions
that this operator be well-defined bounded or invertible is studied.
\par In this paper, we introduce the Gram operator of controlled frames
as a practical tool and discuss the cases in which this operator can
be well-defined, bounded, Hilbert-Schmidt, trace class, compact and
invertible.
\par The content of this paper is as follows: In Section 2, the Gram
operator and Gram matrix of $(U, C)$-controlled frames introduced
and a practical method to diagnose Bessel sequence is given by the
concept of controlled Gram matrix. Also a bounded operator from
$\ell^{1}(\mathbb{N})$ to $\ell^{\infty}(\mathbb{N})$ is achieved
with the assumption that the controlled Gram operator is
well-defined and bounded on $\ell^{2}(\mathbb{N})$. In Section 3,
the general construction of $(U, C)$-controlled Riesz bases proposed
and an equivalent feasible method for $\{f_{k}\}_{k=1}^{\infty}$ to
being a $(U, C)$-controlled Riesz basis, given. Throughout this
paper, $H$ is a separable Hilbert space and $GL(H)$ is the space of all
bounded and invertible operators on $H$ and $GL^{+}(H)$ is the space
of all bounded, invertible and positive operators on $H$. Also $U,
C\in GL(H)$.
\section{Gram matrix of controlled frames}
Controlled frames with one and two controller operators were first
introduced in \cite{10} and \cite{11}, respectively. They are
equivalent to standard frames and so this concept gives a
generalization way to check the frame conditions.
\begin{defn}
Let $\{f_{k}\}_{k\in I}$ be a sequence of vectors in a Hilbert space
$\mathcal{H}$ and $U, C\in GL(\mathcal{H})$. Then $\{f_{k}\}_{k\in
I}$ is called a frame controlled by $U$ and $C$ or $(U,
C)$-controlled frame if there exist two constants $0<A\leq
B<\infty,$ such that
\begin{equation}
A\|f\|^{2}\leq\sum_{k\in I}\langle f, Uf_{k}\rangle\langle Cf_{k},
f\rangle\leq B\|f\|^{2},\ \ f\in \mathcal{H}.
\end{equation}
If only the right inequality holds, then we call $\{f_{k}\}_{k\in
I}$ a $(U, C)$-controlled Bessel sequence. If $A=B$ then
$\{f_{k}\}_{k\in I}$ is called a $(U, C)$-controlled tight frame.
\end{defn}
Let $F=\{f_{k}\}_{k\in I}$ be a Bessel sequence of elements in $H$.
We define the synthesis operator $T_{UF}$,
\begin{equation*}
T_{UF}:\ell^{2}(I)\rightarrow H,\ \ T_{UF}(\{a_{k}\}_{k\in
I})=\sum_{k\in I}a_{k}Uf_{k},\ \  \{a_{k}\}_{k\in
I}\in\ell^{2}(I),
\end{equation*}
and the adjoint operator $T_{UF}^{*}$ which is called the analysis
operator is as follows:
\begin{equation*}
T_{UF}^{*}:H\rightarrow\ell^{2}(I),\ \ T_{UF}^{*}=\{\langle f,
Uf_{k}\rangle\}_{k\in I}
\end{equation*}
Now we define the controlled frame operator $S_{UC}$ on $H$
\begin{equation*}
S_{UC}f=T_{CF}T_{UF}^{*}f=\sum_{k\in I}\langle f, Uf_{k}\rangle
Cf_{k},\ \  f\in H.
\end{equation*}
It is easy to see that if $F=\{f_{k}\}_{k\in I}$ is a $(U,
C)$-controlled frame with bounds $A_{UC}$ and $B_{UC}$, then $S_{UC}$ is
well-defined and
$$A_{UC}Id_{H}\leq S_{UC}\leq B_{UC}Id_{H}.$$
Hence $S_{UC}$ is a bounded, invertible, self-adjoint and positive
linear operator. Therefore, we have $S_{UC}=S_{UC}^{*}=S_{CU}$
\cite{10, 11}.

\begin{prop}\cite{11}\label{hosein}
Let $U, C\in GL(H)$ and $F$ be a family of vectors in a Hilbert
space $H$. Then the following statements hold:
\begin{enumerate}
\item If $F$ is a $(U, C)$-controlled frame for $H$. Then $F$ is a
frame for $H$.
\item If $F$ is a frame for $H$ and $CS_{F}U^{*}$ is a positive
operator, then $F$ is a $(U, C)$-controlled frame for $H$.
\end{enumerate}
\end{prop}
By the above proposition for a frame $F$ which is also a $(U,
C)$-controlled frame for $H$, we have
$$CS_{F}U^{*}=S_{UC}=S^{*}_{UC}=US_{F}C^{*}=S_{CU}.$$
Also we have new reconstruction formula as follows:
$$f=\sum_{i\in I}\langle f, Uf_{i}\rangle
S^{-1}_{UC}Cf_{i}=\sum_{i\in I}\langle f, S^{-1}_{UC}Uf_{i}\rangle
Cf_{i},\ \ f\in H.$$
\begin{prop}\label{12}\cite{12}
Let $T:\mathcal{H}\rightarrow \mathcal{H}$ be a linear operator.
Then the following conditions are equivalent:
\begin{enumerate}
\item  There exist $A>0$ and $B<\infty$, such that $AI\leq T\leq BI$;
\item $T$ is positive and there exist $A>0$ and $B<\infty$, such that
$A\|f\|^{2}\leq\|T^{\frac{1}{2}}f\|^{2}\leq B\|f\|^{2}$;
\item $T$ is positive and $T^{\frac{1}{2}}\in GL(\mathcal{H})$.
\item There exists a self-adjoint operator $S\in GL(\mathcal{H})$, such that
$S^{2}=T$;
\item $T\in GL^{(+)}(\mathcal{H})$;
\item There exist constants $A>0$ and $B<\infty$ and an operator
$C\in GL^{(+)}(\mathcal{H})$ such that $AC\leq T\leq BC$;
\item For every $C\in GL^{(+)}(\mathcal{H})$, there exist constants $A>0$ and
$B<\infty$ such that $AC\leq T\leq BC.$
\end{enumerate}
\end{prop}
Since controlled frames and standard frames are equivalent in some
cases, we define the Gram matrix of controlled frames as an
effective tool to diagnose the controlled Bessel, frame or Riesz
bases. But as we see the results are not always the same as
cross-Gram matrix or Gram matrix of standard frames.
\par If $\{f_{k}\}_{k=1}^{\infty}$ is a $(U, C)$-controlled Bessel
sequence, we can compose the synthesis operator $T_{CF}$ and
$T_{UF}^{*}$, so we obtain the bounded operator
$$T_{UF}^{*}T_{CF}:\ell^{2}(\mathbb{N})\rightarrow\ell^{2}(\mathbb{N}),\ \ T_{UF}^{*}T_{CF}\{a_{k}\}_{k=1}^{\infty}
=\{\langle\sum_{j=1}^{\infty}a_{j}Cf_{j},
Uf_{k}\rangle\}_{k=1}^{\infty}.$$ We call this operator the $(U,
C)$-controlled Gram operator.
\\Suppose that
$\{e_{k}\}_{k=1}^{\infty}$ is the canonical orthonormal basis for
$\ell^{2}(\mathbb{N})$, the $jk$-th entry in the matrix representation
of $T_{UF}^{*}T_{CF}$ is
$$T_{UF}^{*}T_{CF}=\{\langle Cf_{j}, Uf_{k}\rangle\}_{k, j=1}^{\infty}.$$
The matrix $\{\langle Cf_{j}, Uf_{k}\rangle\}_{k, j=1}^{\infty}$ is
called the Gram matrix associated to $(U, C)$-controlled Bessel
sequence $\{f_{k}\}_{k=1}^{\infty}$ or $(U, C)$-controlled Gram
matrix associated to $\{f_{k}\}_{k=1}^{\infty}$.
\begin{rem}
The above argument shows that if $\{f_{k}\}_{k=1}^{\infty}$ is a
$(U, C)$-controlled Bessel sequence, the $(U, C)$-controlled Gram
matrix associated to $\{f_{k}\}_{k=1}^{\infty}$ is well-defined and
bounded.
\end{rem}
\begin{exam}
Let $\{e_{k}\}_{k=1}^{\infty}$ be the canonical orthonormal basis
for $\ell^{2}(\mathbb{N})$. Consider the sequence
$f_{2k+1}=e_{2k+1}-e_{2k+2},\ k=0, 1, 2, ...$ and
$f_{2k}=e_{2k-1}+e_{2k},\ k=1, 2, 3,...$. If we define the operators
\begin{equation*}
C:\ell^{2}(\mathbb{N})\rightarrow\ell^{2}(\mathbb{N}),\ \ C(x_{1},
x_{2}, x_{3}, x_{4},...)=(-x_{1}, x_{2}, -x_{3}, x_{4},...)
\end{equation*} and 
\begin{equation*}
U:\ell^{2}(\mathbb{N})\rightarrow\ell^{2}(\mathbb{N}),\ \ U(x_{1},
x_{2}, x_{3}, x_{4},...)=(x_{1}, -x_{2}, x_{3}, -x_{4},...).
\end{equation*}
Then a straight calculation shows that $\{f_{k}\}_{k=1}^{\infty}$ is
a $(U, C)$-controlled tight frame for $\ell^{2}(\mathbb{N})$ with
bound $2$ and the $(U, C)$-controlled Gram matrix associated to
$\{\langle Cf_{k}, Uf_{j}\rangle\}_{j, k=1}^{\infty}$ is
well-defined and bounded on $\ell^{2}(\mathbb{N})$ with bound $2$.
\end{exam}
In Example 2.1 of \cite{7}, we saw that although the cross-Gram
matrix associated to $\{\langle f_{k}, g_{j}\rangle\}_{j,
k=1}^{\infty}$ is well-defined and bounded,
$\{f_{k}\}_{k=1}^{\infty}$ is not a Bessel sequence. Now a logical
question is that: can we say the sequence $\{f_{k}\}_{k=1}^{\infty}$
is Bessel if the $(U, C)$-controlled Gram matrix associated to
$\{f_{k}\}_{k=1}^{\infty}$ is well-defined and bounded? The
following lemma shows that the answer is positive.
\begin{lem}
Suppose that $U, C\in GL(H)$ and the $(U, C)$-controlled Gram matrix
associated to $\{f_{k}\}_{k=1}^{\infty}$ is well-defined and
bounded. Then $\{f_{k}\}_{k=1}^{\infty}$ is a Bessel sequence.
\end{lem}
\begin{proof}
By assumption, there exists $M>0$ such that
\begin{equation}\label{fot}
\sum_{k=1}^{\infty}|\sum_{j=1}^{\infty}c_{j}\langle Cf_{j},
Uf_{k}\rangle|^{2}\leq M\sum_{k=1}^{\infty}|c_{k}|^{2},\ \
\{c_{k}\}_{k=1}^{\infty}\in\ell^{2}(\mathbb{N}).
\end{equation}
Consider $\{c_{k}\}_{k=1}^{\infty}=(0,...,1,0,...)$, we get
\begin{equation}
\sum_{k=1}^{\infty}|\langle Cf_{J}, Uf_{k}\rangle|^{2}\leq M,\ \
 J\in\mathbb{N}.
\end{equation}
or
\begin{equation}\label{mirza}
\sum_{k=1}^{\infty}|\langle U^{*}Cf_{J}, f_{k}\rangle|^{2}\leq M,\ \
 J\in\mathbb{N}.
\end{equation}
Suppose that $\{f_{k}\}_{k=1}^{\infty}$ is not a Bessel sequence,
then for all integer $N>0$ there exists $g_{N}\in H$ such that
\begin{equation}\label{maryam}
\sum_{k=1}^{\infty}|\langle g_{N}, f_{k}\rangle|^{2}>N\|g_{N}\|^{2}.
\end{equation}
Therefore three cases may happen:
\\ Case 1. If
$g_{N}\in\{U^{*}Cf_{k}\}_{k=1}^{\infty}$, then there exists
$j\in\mathbb{N}$ such that $U^{*}Cf_{j}=g_{N}$. Therefore by
(\ref{maryam}), we have
\begin{equation}
\sum_{k=1}^{\infty}|\langle U^{*}Cf_{j},
f_{k}\rangle|^{2}>N\|g_{N}\|^{2},
\end{equation}
which is a contradiction with (\ref{mirza}).
\\Case 2. If
$g_{N}\in\overline{span}\{U^{*}Cf_{k}\}_{k=1}^{\infty}$. Then
\begin{equation}
\sum_{k=1}^{\infty}|\langle\sum_{j=1}^{\infty}c_{j}U^{*}Cf_{j},
f_{k}\rangle|^{2}>N\|g_{N}\|^{2}
\end{equation}
or
\begin{equation}\label{khan}
\sum_{k=1}^{\infty}|\sum_{j=1}^{\infty}c_{j}\langle Cf_{j},
Uf_{k}\rangle|^{2}>N\|g_{N}\|^{2},
\end{equation}
which is a contradiction with (\ref{fot}).
\\Case 3. If $g_{N}\notin\overline{span}\{U^{*}Cf_{k}\}_{k=1}^{\infty}$. Consider
$M=\overline{span}\{U^{*}Cf_{k}\}_{k=1}^{\infty}$. Then, we can write
$g_{N}=p_{N}+h_{N}$, where $p_{N}\in M$ and $h_{N}\in M^{\perp}$,
$h_{N}\neq 0$, where $M^{\bot}$ is the orthogonal complement of $M$
in $H$. Now, we have
\begin{equation}
\sum_{k=1}^{\infty}|\langle g_{N},
f_{k}\rangle|^{2}=\sum_{k=1}^{\infty}|\langle p_{N}+h_{N},
f_{k}\rangle|^{2}=\sum_{k=1}^{\infty}\langle p_{N},
f_{k}\rangle|^{2}>N\|g_{N}\|^{2},
\end{equation}
Which is a contradiction like case 2. Therefore
$\{f_{k}\}_{k=1}^{\infty}$ is a Bessel sequence.
\end{proof}
\begin{lem}\label{nabat}
Let $\{f_{k}\}_{k=1}^{\infty}$ be a $(U, C)$-controlled Bessel
sequence, then $\{f_{k}\}_{k=1}^{\infty}$ is a Bessel sequence in
$\mathcal{H}$.
\end{lem}
\begin{proof}
Let $S_{UC}$ be the frame operator of $\{f_{k}\}_{k=1}^{\infty}$.
Define $S_{F}=C^{-1}S_{UC}(U^{*})^{-1}$. Since $S_{UC}$, $U$ and $C$
are bounded operators, $S_{F}$ is well-defined and bounded.
Therefore, there exists $B>0$ such that
$$\sum_{k=1}^{\infty}|\langle f, f_{k}\rangle|^{2}=\langle S_{F}f,
f\rangle\leq B\|f\|^{2}.$$
\end{proof}
\begin{defn}\cite{15}
Suppose that $E$ is an orthonormal basis for $H$. A bounded operator
$T\in B(H)$ is called a Hilbert-Schmidt operator if
\begin{equation*}
\|T\|_{2}=\sqrt{\sum_{x\in E}\|Tx\|^{2}}<\infty
\end{equation*}
\end{defn}
\begin{defn}\cite{15}
Suppose that $E$ is an orthonormal basis for $H$. A bounded operator
$T\in B(H)$ is called a trace-class operator if
\begin{equation*}
\|T\|_{1}=\sum_{x\in E}\langle |T|(x), x\rangle<\infty
\end{equation*}
\end{defn}
We denote the class of all Hilbert-Schmidt operators on $H$ and the
class of trace-class operators on $H$ by $L^{2}(H)$ and $L^{1}(H)$,
respectively. In \cite{15}, we see that $L^{1}(H)\subseteq L^{2}(H)$.
\begin{thm}(Polar Decomposition)\cite{15}\label{moha}
Let $V$ be a bounded linear operator on $H$. Then there is a
unique partial isometry $U\in B(H)$ such that
$$V=U|V| ,\ \ \ \  \ker(U)=\ker(V).$$
Moreover, $U^{*}V=|V|.$
\end{thm}
\begin{thm}
Suppose that $U, C\in GL(H)$. Let $F=\{f_{k}\}_{k=1}^{\infty}$ be a
$(U, C)$-controlled frame and $G_{CU}$ is the $(U, C)$-controlled
Gram operator associated to $\{f_{k}\}_{k=1}^{\infty}$. Then
\begin{enumerate}
\item 
$G_{CU}$ is a Hilbert-Schmidt operator if and only if $H$ is finite
dimensional.
\item $G_{CU}$ is a trace-class operator if and only if $H$ is finite
dimensional.
\end{enumerate}
\end{thm}
\begin{proof}
\begin{enumerate}
\item  Suppose that $dim H<\infty$. By Lemma \ref{nabat},
$\{f_{k}\}_{k=1}^{\infty}$ is a Bessel sequence. Therefore there
exists $B<\infty$ such that
\begin{eqnarray}
\nonumber
\|G_{CU}\|_{2}^{2}&=&\sum_{j=1}^{\infty}\sum_{k=1}^{\infty}|\langle
Cf_{j},
Uf_{k}\rangle|^{2}=\sum_{j=1}^{\infty}\sum_{k=1}^{\infty}|\langle
f_{j},
C^{*}Uf_{k}\rangle|^{2}\\&=&\sum_{k=1}^{\infty}\sum_{j=1}^{\infty}|\langle
f_{j}, C^{*}Uf_{k}\rangle|^{2}\leq
B\sum_{k=1}^{\infty}\|C^{*}Uf_{k}\|^{2}\leq
B\|C^{*}\|^{2}\|U\|^{2}\sum_{k=1}^{\infty}\|f_{k}\|^{2}.
\end{eqnarray}
Therefore by Proposition 5.1. in \cite{14}, $G_{CU}$ is a
Hilbert-Schmidt operator.
\\Now suppose that $G_{CU}$ is a Hilbert-Schmidt operator. By Proposition
\ref{hosein}, $\{f_{k}\}_{k=1}^{\infty}$ is a frame for $H$,
therefore there exists $A>0$ such that
\begin{eqnarray*}
A\|(C^{*}U)^{-1}\|^{-2}\sum_{k=1}^{\infty}\|f_{k}\|^{2}&\leq&\sum_{k=1}^{\infty}\sum_{j=1}^{\infty}|\langle
Cf_{j},
Uf_{k}\rangle|^{2}\\&=&\sum_{j=1}^{\infty}\sum_{k=1}^{\infty}|\langle
Cf_{j}, Uf_{k}\rangle|^{2}=\|G_{CU}\|_{2}^{2}.
\end{eqnarray*}
The proof is evident by Proposition 5.1. in \cite{14}.
\item Suppose that $dim H<\infty$. By polar decomposition,
there is a unique partial isometry $M\in B(\ell^{2}(\mathbb{N}))$
such that $|G_{CU}|=M^{*}G_{CU}$ and $G_{CU}=M|G_{CU}|$. Therefore
\begin{equation}\label{safar}
|G_{CU}|=M^{*}T^{*}_{UF}T_{CF}.
\end{equation}
Now we show that $T_{CF}$ is a Hilbert-Schmidt operator. Suppose
that $\{e_{k}\}_{k=1}^{\infty}$ is the canonical orthonormal basis
for $\ell^{2}(\mathbb{N})$. So we have
\begin{eqnarray*}
\|T_{CF}\|_{2}^{2}=\sum_{\{e_{k}\}_{k=1}^{\infty}}
\|T_{CF}(e_{k})\|^{2}=\sum_{k=1}^{\infty}\|Cf_{k}\|^{2}\leq\|C\|^{2}
\sum_{k=1}^{\infty}\|f_{k}\|^{2}.
\end{eqnarray*}
Therefore by Proposition 5.1. in \cite{14}, $T_{CF}$ is a
Hilbert-Schmidt operator. Since $\|T_{UF}\|_{2}=\|T_{UF}^{*}\|_{2}$,
we deduce that $T_{UF}^{*}$ is also a Hilbert-Schmidt operator. Now
by Theorems 2.4.10. and 2.4.13. in \cite{15} and (\ref{safar}), $|G_{CU}|$
is a trace-class operator. Since $G_{CU}=M|G_{CU}|$, by Theorem
2.4.15 in \cite{15}, $G_{CU}$ is a trace-class operator.
\\Vice versa, let $G_{CU}$ is a trace-class operator. Since $L^{1}(H)\subseteq
L^{2}(H)$, we deduce that $G_{CU}$ is a Hilbert-Schmidt operator and
so $H$ is a finite dimensional space by part (1).
\end{enumerate}
\end{proof}
\begin{cor}
If $H$ is finite dimensional and $\{f_{k}\}_{k=1}^{\infty}$ is a
$(U, C)$-controlled frame. Then $G_{CU}$ is a compact operator.
\end{cor}
The following proposition gives a well-defined and bounded operator
from $\ell^{1}(\mathbb{N})$ to $\ell^{\infty}(\mathbb{N})$  when the
$(U, C)$-controlled Gram matrix is well-defined and bounded on
$\ell^{2}(\mathbb{N})$.
\begin{prop}
Suppose that the $(U, C)$-controlled Gram matrix associated to
$\{f_{k}\}_{k=1}^{\infty}$, is well-defined and bounded. Then a
bounded operator can be defined from $\ell^{1}(\mathbb{N})$ to
$\ell^{\infty}(\mathbb{N})$.
\end{prop}
\begin{proof}
Suppose that $G_{CU}$ is the operator associated to the matrix
$\{\langle Cf_{j}, Uf_{k}\rangle\}_{k, j=1}^{\infty}$. Since
$G_{CU}$ is well-defined and bounded on $\ell^{2}(\mathbb{N})$, for
$\{c_{k}\}_{k=1}^{\infty}\in \ell^{2}(\mathbb{N})$, there exists
$B>0$ such that
\begin{equation}
\sum_{k=1}^{\infty}|\sum_{j=1}^{\infty}c_{j}\langle Cf_{j},
Uf_{k}\rangle|^{2}\leq B\sum_{k=1}^{\infty}|c_{k}|^{2}.
\end{equation}
Therefore for each $j\in \mathbb{N}$,
\begin{equation}\label{sa}
\sum_{k=1}^{\infty}|\langle Cf_{j}, Uf_{k}\rangle|^{2}\leq B.
\end{equation}
Consider $M_{j, k}=\langle Cf_{j}, Uf_{k}\rangle$ and $M=\{M_{j,
k}\}_{j, k=1}^{\infty}$. Then
$M\{c_{k}\}_{k=1}^{\infty}=\{\sum_{j=1}^{\infty}M_{k,
j}c_{j}\}_{k=1}^{\infty}$. Now, we show that $M$ defines a
well-defined and bounded operator from $\ell^{1}(\mathbb{N})$ to
$\ell^{\infty}(\mathbb{N})$.
\\First, we show that $\sum_{j=1}^{\infty}M_{k, j}c_{j}$ is convergent for each $k\in\mathbb{N}$.
Given arbitrary $n, m\in\mathbb{N}$, $n\geq m$
\begin{eqnarray*}
|\sum_{j=m+1}^{n}M_{k, j}c_{j}|^{2}\leq(\sum_{j=m+1}^{n}|M_{k,
j}||c_{j}|)^{2}\leq(\sum_{j=m+1}^{n}|M_{k,
j}|^{2})(\sum_{j=m+1}^{n}|c_{j}|^{2}).
\end{eqnarray*}
By (\ref{sa}) and since $\ell^{1}(\mathbb{N})\subseteq
\ell^{2}(\mathbb{N})$, we get the result. Now, we show that $M$ is a
bounded operator. For $\{c_{k}\}_{k=1}^{\infty}\in
\ell^{1}(\mathbb{N})$, we have
\begin{eqnarray*}
\|M\{c_{k}\}_{k=1}^{\infty}\|_{\infty}^{2}&=&\|\{\sum_{j=1}^{\infty}M_{k,
j}c_{j}\}_{k=1}^{\infty}\|_{\infty}^{2}=\sup_{k\in\mathbb{N}}|\sum_{j=1}^{\infty}M_{k,
j}c_{j}|^{2}\\&\leq&\sum_{j=1}^{\infty}|c_{j}|^{2}\sup_{k\in\mathbb{N}}(\sum_{j=1}^{\infty}|M_{k,
j}|^{2})\\&\leq& B\sum_{j=1}^{\infty}|c_{j}|^{2}\leq
B\sum_{j=1}^{\infty}|c_{j}|.
\end{eqnarray*}
\end{proof}
\section{$(U, C)$-controlled Riesz basis}
In this section, we propose a clear structure of $(U, C)$-controlled
Riesz basis and show that every $(U, C)$-controlled Riesz basis is a
$(U, C)$-controlled frame. Also an equivalent condition for a
sequence $\{f_{k}\}_{k=1}^{\infty}$ being controlled Riesz basis
given.
\begin{defn}
Suppose that $\{e_{k}\}_{k=1}^{\infty}$ is an orthonormal basis for
$\mathcal{H}$. A $(U, C)$-controlled Riesz basis for $\mathcal{H}$
is a family of the form $\{U^{-1}CMe_{k}\}_{k=1}^{\infty}$, where
$M$ is a bounded bijective operator on $H$.
\end{defn}
\begin{cor}
Every $(U, C)$-controlled Riesz basis is a Riesz basis for $H$.
\end{cor}
\begin{lem}\label{inv}
Suppose that $U$ is a positive invertible operator on a Hilbert
space $H$. Then $U^{-1}$ is positive.
\end{lem}
\begin{proof}
Since $U$ is an invertible operator, for each $x\in H$,
there exists $y\in H$ such that $Uy=x$. So
\begin{equation*}
\langle U^{-1}x, x\rangle=\langle U^{-1}Uy, Uy\rangle=\langle y,
Uy\rangle\geq 0
\end{equation*}
\end{proof}
\begin{lem}\label{com}\cite{13}
If two bounded self-adjoint linear operators $S$ and $T$ on a
Hilbert space $H$ are positive and commute, then their
product $ST$ is positive.
\end{lem}
\begin{thm}
Suppose that $\{f_{k}\}_{k=1}^{\infty}$ is a $(U, C)$-controlled
Riesz-basis for $H$. Assume that $U, C\in GL^{+}(H)$ and
$U^{-1}$ and $C$ commute. Then $\{f_{k}\}_{k=1}^{\infty}$ is a $(U,
C)$-controlled frame.
\end{thm}
\begin{proof}
Since $\{f_{k}\}_{k=1}^{\infty}$ is a $(U, C)$-controlled
Riesz-basis, for each $f\in H$, we have
\begin{eqnarray}\label{11}
\nonumber\sum_{k=1}^{\infty}\langle f, Uf_{k}\rangle\langle Cf_{k},
f\rangle &=&\sum_{k=1}^{\infty}\langle f, CMe_{k}\rangle\langle
CU^{-1}CMe_{k}, f\rangle\\&=&\sum_{k=1}^{\infty}\langle f,
CMe_{k}\rangle\langle CMe_{k}, (U^{-1})^{*}C^{*}f\rangle.
\end{eqnarray}
Consider $g_{k}=CMe_{k},$ for each $k\in\mathbb{N}$. Then
$\{g_{k}\}_{k=1}^{\infty}$ is a Riesz basis for $H$. So by
(\ref{11}), we have
\begin{eqnarray*}
\sum_{k=1}^{\infty}\langle f, Uf_{k}\rangle\langle Cf_{k}, f\rangle
&=&\langle T_{g_{k}}T^{*}_{g_{k}}f, (U^{-1})^{*}C^{*}f\rangle\\&=&
\langle S_{g_{k}}f, (U^{-1})^{*}C^{*}f\rangle\\&=&\langle
CU^{-1}S_{g_{k}}f, f\rangle.
\end{eqnarray*}
Since $U, C\in GL^{+}(\mathcal{H})$ and $U^{-1}$ and $C$ commute by
Lemma \ref{inv} and \ref{com}, $CU^{-1}\in GL^{+}(H)$. Since
$S_{g_{k}}\in GL^{+}(H)$, by Proposition \ref{12}, there exist $A>0$
and $B<\infty$, such that $AI\leq CU^{-1}S_{g_{k}}\leq BI$.
\end{proof}
\begin{thm}\label{riesz}
If $\{f_{k}\}_{k=1}^{\infty}$ is a $(U, C)$-controlled Riesz basis
for $H$, then $\{f_{k}\}_{k=1}^{\infty}$ is a Bessel sequence.
Furthermore, there exists a unique controlled Riesz basis sequence
$\{g_{k}\}_{k=1}^{\infty}$ for $H$ such that for any
$f\in\mathcal{H}$
\begin{enumerate}
\item
$f=\sum_{k=1}^{\infty}\langle f, g_{k}\rangle
f_{k}=\sum_{k=1}^{\infty}\langle f, f_{k}\rangle g_{k}.$
\item $f=\sum_{k=1}^{\infty}\langle f, Ug_{k}\rangle Cf_{k}=
\sum_{k=1}^{\infty}\langle f, Cf_{k}\rangle Ug_{k}.$
\end{enumerate}
\end{thm}
\begin{proof}
\begin{enumerate}
\item Suppose that $\{e_{k}\}_{k=1}^{\infty}$ is an orthonormal basis
for $\mathcal{H}$. Since $\{f_{k}\}_{k=1}^{\infty}$ is a $(U,
C)$-controlled Riesz basis, there exists a bounded bijective
operator $M$ on $H$ such that $f_{k}=U^{-1}CMe_{k}$ for each
$k\in\mathbb{N}$. So we have
\begin{equation*}
M^{-1}C^{-1}f=\sum_{k=1}^{\infty}\langle M^{-1}C^{-1}f, e_{k}\rangle
e_{k},\ \ f\in \mathcal{H},
\end{equation*}
so
\begin{equation*}
f=\sum_{k=1}^{\infty}\langle f, (C^{-1})^{*}(M^{-1})^{*}e_{k}\rangle
CMe_{k},\ \ f\in \mathcal{H},
\end{equation*}
therefore
\begin{equation*}
U^{-1}f=\sum_{k=1}^{\infty}\langle f,
(C^{-1})^{*}(M^{-1})^{*}e_{k}\rangle U^{-1}CMe_{k},\ \ f\in
\mathcal{H},
\end{equation*}
and
\begin{equation*}
f=\sum_{k=1}^{\infty}\langle f,
U^{*}(C^{-1})^{*}(M^{-1})^{*}e_{k}\rangle f_{k},\ \ f\in
\mathcal{H}.
\end{equation*}
Therefore by considering $g_{k}=U^{*}(C^{-1})^{*}(M^{-1})^{*}e_{k}$,
$\{g_{k}\}_{k=1}^{\infty}$ is a $((U^{*})^{-1},
(C^{*})^{-1})$-controlled Riesz basis. A simple calculation shows
that $\{g_{k}\}_{k=1}^{\infty}$ is a unique sequence that satisfies
in (1).
\item Considering
$g_{k}=U^{-1}(C^{-1})^{*}U^{*}(C^{-1})^{*}(M^{-1})^{*}e_{k},$ we get
a unique $(C^{*}U, U^{*}(C^{-1})^{*})$-controlled Riesz basis, which
satisfies in (2).
\end{enumerate}

\end{proof}
\begin{cor}\label{mum}
If the sequences $\{f_{k}\}_{k=1}^{\infty}$ and
$\{g_{k}\}_{k=1}^{\infty}$ satisfy in part (2) of Theorem
\ref{riesz}, then $\langle Cf_{k}, Ug_{j}\rangle=\delta_{k, j}$.
\end{cor}
Analogous to Theorem 3.6.6. of \cite{12}, the following theorem gives an equivalent and practical condition
for a sequence $\{f_{k}\}_{k=1}^{\infty}$ being a $(U,
C)$-controlled Riesz basis.
\begin{thm}\label{hams}
Suppose that $U, C\in GL^{+}(H)$. Assume that $U$ and $U^{-1}$
commute with $C$. For a sequence $\{f_{k}\}_{k=1}^{\infty}$ in
$H$, the following conditions are equivalent:
\begin{enumerate}
\item $\{f_{k}\}_{k=1}^{\infty}$ is a $(U, C)$-controlled Riesz
basis for $H$.
\item $\{f_{k}\}_{k=1}^{\infty}$ is complete in $H$, and there
exist constants $L, P>0$ such that for
$\{c_{k}\}_{k=1}^{\infty}\in\ell^{2}(\mathbb{N})$ one has
\begin{equation}\label{sh}
L\sum_{k=1}^{\infty}|c_{k}|^{2}\leq|\langle\sum_{k=1}^{\infty}c_{k}Uf_{k},\sum_{k=1}^{\infty}c_{k}Cf_{k}\rangle|\leq
P\sum_{k=1}^{\infty}|c_{k}|^{2}.
\end{equation}
\end{enumerate}
\end{thm}
\begin{proof}
\begin{enumerate}
\item
$(1)\Rightarrow (2)$. Assume that $\{f_{k}\}_{k=1}^{\infty}$ is a
$(U, C)$-controlled Riesz basis, then there exists a bijective
operator $M$ on $H$ such that $f_{k}=U^{-1}CMe_{k}$, for each
$k\in\mathbb{N}$. By Theorem \ref{riesz}, $\{f_{k}\}_{k=1}^{\infty}$
is complete in $H$. For
$\{c_{k}\}_{k=1}^{\infty}\in\ell^{2}(\mathbb{N})$ we have
\begin{eqnarray}\label{el}
\nonumber|\langle\sum_{k=1}^{\infty}c_{k}Uf_{k},\sum_{k=1}^{\infty}c_{k}Cf_{k}\rangle|&=&|\langle\sum_{k=1}^{\infty}c_{k}CMe_{k},
\sum_{k=1}^{\infty}c_{k}CU^{-1}CMe_{k}\rangle|\\&=&|\langle
CM(\sum_{k=1}^{\infty}c_{k}e_{k}),
CU^{-1}CM(\sum_{k=1}^{\infty}c_{k}e_{k})\rangle|.
\end{eqnarray}
By Lemma \ref{inv} and \ref{com}, $CU^{-1}\in GL^{+}(H)$. So there
exist $A>0$ and $B<\infty$ such that
\begin{eqnarray}\label{mira}
\nonumber A\|CM(\sum_{k=1}^{\infty}c_{k}e_{k})\|^{2}&\leq&|\langle
CM(\sum_{k=1}^{\infty}c_{k}e_{k}),
CU^{-1}CM(\sum_{k=1}^{\infty}c_{k}e_{k})\rangle|\\&\leq&
B\|CM(\sum_{k=1}^{\infty}c_{k}e_{k})\|^{2}.
\end{eqnarray}
Since
\begin{equation*}
B\|CM(\sum_{k=1}^{\infty}c_{k}e_{k})\|^{2}\leq
B\|CM\|^{2}\sum_{k=1}^{\infty}|c_{k}|^{2},
\end{equation*}
and
\begin{equation*}
A\|(CM)^{-1}\|^{-2}\sum_{k=1}^{\infty}|c_{k}|^{2}\|\leq\|CM(\sum_{k=1}^{\infty}c_{k}e_{k})\|^{2},
\end{equation*}
by (\ref{el}) and (\ref{mira}), we deduce that
\begin{eqnarray*}
A\|(CM)^{-1}\|^{-2}\sum_{k=1}^{\infty}|c_{k}|^{2}\leq
|\langle\sum_{k=1}^{\infty}c_{k}Uf_{k},\sum_{k=1}^{\infty}c_{k}Cf_{k}\rangle|\leq
B\|CM\|^{2}\sum_{k=1}^{\infty}|c_{k}|^{2}.
\end{eqnarray*}
So considering $P=B\|CM\|^{2}$ and $L=A\|(CM)^{-1}\|^{-2}$, we get
the proof.
\item $(2)\Rightarrow (1)$ First, we show that
$\{f_{k}\}_{k=1}^{\infty}$ is a Bessel sequence. For this since
$CU\in GL^{+}(H)$, for $\{c_{k}\}\in\ell^{2}(\mathbb{N})$, we have
\begin{eqnarray}\label{ha}
\nonumber|\langle
(CU)(\sum_{k=1}^{\infty}c_{k}f_{k},\sum_{k=1}^{\infty}
c_{k}f_{k})\rangle|&=&|\langle
(CU)^{\frac{1}{2}}(\sum_{k=1}^{\infty} c_{k}f_{k}),(C
U)^{\frac{1}{2}}(\sum_{k=1}^{\infty}c_{k}f_{k})\rangle|\\&=&\|(C
U)^{\frac{1}{2}}(\sum_{k=1}^{\infty}c_{k}f_{k})\|^{2}.
\end{eqnarray}
Now, we show that $\sum_{k=1}^{\infty}c_{k}f_{k}$ is convergent.
Given arbitrary elements $m, n\in\mathbb{N}$, $n>m$, by (\ref{sh})
and (\ref{ha}), for
$\{c_{k}\}_{k=1}^{\infty}\in\ell^{2}(\mathbb{N})$, we have
\begin{eqnarray}
\nonumber\|\sum_{k=m+1}^{n}c_{k}f_{k}\|^{2}&=&\|(C
U)^{\frac{-1}{2}}(CU)^{\frac{1}{2}}
(\sum_{k=m+1}^{n}c_{k}f_{k})\|^{2}\\
\nonumber&\leq& \|(C U)^{\frac{-1}{2}}\|^{2}\|(C
U)^{\frac{1}{2}}(\sum_{k=m+1}^{n}c_{k}f_{k})\|^{2}\\&=&\|(C U)^
{\frac{-1}{2}}\|^{2}P\sum_{k=m+1}^{n}|c_{k}|^{2}
\end{eqnarray}
Therefore $\sum_{k=1}^{\infty}c_{k}f_{k}$ is convergent and
$\{f_{k}\}_{k=1}^{\infty}$ is a Bessel sequence and so
$\{C^{-1}Uf_{k}\}_{k=1}^{\infty}$ is a Bessel sequence. Choose an
orthonormal basis $\{e_{k}\}_{k=1}^{\infty}$ for $H$, and extend by
Lemma 3.3.6 in \cite{13}, the mapping $Me_{k}=C^{-1}Uf_{k}$ to a
bounded operator on $H$. In the same way, extend
$V(C^{-1}Uf_{k})=e_{k}$ to a bounded operator on $H$. Then
$MV=VM=I$, so $M$ is invertible, therefore
$\{f_{k}\}_{k=1}^{\infty}$ is a $(U, C)$-controlled Riesz basis.
\end{enumerate}
\end{proof}
The following theorem gives a practical method to diagnose that
$\{f_{k}\}_{k=1}^{\infty}$ is a $(U, C)$-controlled Riesz basis.
\begin{thm}
Suppose that $U, C\in GL^{+}(H)$. Assume that $U$ and $U^{-1}$
commute with $C$. For a sequence $\{f_{k}\}_{k=1}^{\infty}$ in $H$,
the following conditions are equivalent:
\begin{enumerate}
\item $\{f_{k}\}_{k=1}^{\infty}$ is a $(U, C)$-controlled Riesz
basis.
\item $\{f_{k}\}_{k=1}^{\infty}$ is complete and it's controlled-Gram matrix
$\{\langle Cf_{k}, Uf_{j}\rangle\}_{j, k=1}^{\infty}$ defines a
bounded, invertible operator on $\ell^{2}(\mathbb{N})$.
\item $\{f_{k}\}_{k=1}^{\infty}$ is complete, $(U, C)$-controlled Bessel
sequence and has a complete biorthogonal sequence that is also a
$(U, C)$-controlled Bessel sequence.
\end{enumerate}
\end{thm}
\begin{proof}
By a similar calculation of Theorem 3.4.4 in \cite{13}, Corollary
\ref{mum} and Theorem \ref{hams}, we get the proof.
\end{proof}


\bibliographystyle{plain}

\end{document}